\documentclass[reqno]{amsart}

\usepackage{cite,amsmath,amssymb}
\usepackage{hyperref}

\hypersetup{
  colorlinks=true,
  linkcolor=blue,
  citecolor=blue,
  urlcolor=blue,
  pdftitle={A logarithmically monotone symmetric norm which is not fully symmetric}
}

\newtheorem{theorem}{Theorem}[section]
\newtheorem{proposition}[theorem]{Proposition}
\newtheorem{lemma}[theorem]{Lemma}
\newtheorem{corollary}[theorem]{Corollary}
\theoremstyle{definition}
\newtheorem{definition}[theorem]{Definition}
\theoremstyle{remark}

\numberwithin{equation}{section}

\newcommand{\Rplus}{(0,\infty)}
\newcommand{\lsub}{\prec\!\prec_{\log}}
\newcommand{\hsub}{\prec\!\prec}
\newcommand{\K}{K}

\title[Logarithmically monotone symmetric norms]{A logarithmically monotone symmetric norm which is not fully symmetric}
 
\author{Jinghao Huang}\thanks{J. Huang was supported by the NNSF of China  (No. 12031004, 12301160 and  12471134).
	}
	\address{Institute for  Advanced Study in  Mathematics, Harbin Institute of Technology, Harbin, 150001, China}	\email{{\color{blue}jinghao.huang@hit.edu.cn}}

\begin{document}

\begin{abstract}
We construct an explicit symmetric Banach norm on a function space over $(0,\infty)$ which is monotone with respect to logarithmic submajorization but is not strongly symmetric, and hence is not fully symmetric. 
This answers a question raised by Dodds, Dodds, Sukochev, and Zanin.
We also construct a strongly symmetric Banach function space over
$(0,\infty)$ which is monotone with respect to logarithmic
submajorization and admits no equivalent fully symmetric norm.
\end{abstract}

\maketitle

\section{Introduction}

Let $E$ be a Banach lattice of measurable functions on $\Rplus$.  Two
conventions for function spaces coexist in the literature.  In the
terminology of Kre\u{\i}n, Petunin, and Semenov
\cite[Chapter~II, \S4]{KPS}, a symmetric space is a complete ideal space
whose norm is invariant under equimeasurability; the Fatou property is a
separate property and is not part of the definition.  By contrast, the
definition of a Banach function norm in Bennett and Sharpley includes the
Fatou property \cite[Chapter~I, \S1, Definition~1.1(P3)]{BS}.  In the
rearrangement-invariant setting, such a norm is fully symmetric, that is,
monotone with respect to Hardy--Littlewood submajorization; see
  \cite[Chapter~II, Theorem~4.6]{BS} and also
  \cite[Corollary~5.1.12]{DPS}.  We follow the
broader Kre\u{\i}n--Petunin--Semenov convention.

The distinction is substantive.  Sedaev, Semenov, and Sukochev constructed fully symmetric function spaces which admit no equivalent Fatou norm, as well as symmetric function spaces which admit no equivalent strongly symmetric norm~\cite{SSS}.  Their examples show that the absence of the Fatou property cannot in general be repaired by an equivalent renorming. 
The construction in this paper has a different purpose and is less extreme in this respect: its underlying Marcinkiewicz function space $M_\psi$ carries its standard fully symmetric Fatou norm, while the particular equivalent norm constructed here is neither strongly symmetric nor Fatou.

Logarithmic (Weyl) submajorization is a strictly stronger order relation than the
Hardy--Littlewood submajorization.  Dodds, Dodds, Sukochev, and Zanin
developed its continuous and noncommutative forms in \cite{DDSZ}.  We quote the following from 
\cite[Remark~5.13]{DDSZ}:
\begin{quote}
    While it is well-known \cite{SSS} that there are symmetric spaces on the $[0,\infty )$
 which are not closed subspaces of any fully symmetric space, we do not have an explicit example of a symmetric space on $[0,\infty )$
 whose norm is monotone with respect to logarithmic submajorization but which is not fully symmetric. This seems to be a difficult problem.
\end{quote}
Sukochev and Zanin subsequently proved that every symmetric
Banach sequence space admits an equivalent logarithmically monotone
symmetric norm \cite{SZ}, whose proof relies on an essentially
noncommutative triangular-operator argument.
This, in particular, yields the existence of a norm on a sequence space 
which is 
monotone with respect to logarithmic submajorization but is not fully symmetric. 
However, the case for function spaces remains open.

Our construction is motivated by several  ideas used in \cite{CRSS} and \cite{KSRI,KSNO}.
  Carey, Rennie, Sedaev, and Sukochev studied an endpoint zeta seminorm defined through global $L^p$ moments as $p\downarrow1$~\cite[Section~4.2]{CRSS}. 
  Kalton and Sukochev constructed bounded positive symmetric singular linear functionals on Marcinkiewicz spaces which fail Hardy--Littlewood monotonicity \cite{KSRI}. 
  They observed that adding such a functional to the standard Marcinkiewicz norm produces an equivalent symmetric norm which is not fully symmetric \cite[p.~83]{KSNO}.  
 The norms constructed in the present paper can also be viewed as a
 ``standard Marcinkiewicz norm plus a singular term''.
However, our construction of functionals is different from that in \cite{KSNO}.

 The first main result of this note is the following.

\begin{theorem}\label{thm:main}
There are a symmetric function space $E$ over $\Rplus$ and a symmetric
Banach norm $N$ on $E$ such that, for all $u,v\in E$,
\[
  u\lsub v \quad\Longrightarrow\quad N(u)\le N(v),
\]
but $N$ is not strongly symmetric.  More precisely, there are nonnegative decreasing functions $f,g$ such that
\[
  g\hsub f
  \qquad\text{and}\qquad
  N(g)>N(f).
\]
The norm $N$ is equivalent to the standard norm of a Marcinkiewicz space, but it does not have the Fatou property.
\end{theorem}

It is well-known that the absolute kernel of a positive functional on a symmetric function space 
is an order ideal, see, e.g. \cite[p. 6436]{DP}. 
The second example  presents a  continuous symmetric
lattice seminorm $\Phi$ on a classical Marcinkiewicz space $M_\psi$ such that the
kernel
\[
  X=\ker\Phi
\]
is closed under logarithmic submajorization but not under
Hardy--Littlewood submajorization.  

 In the paper \cite{SSS} of   Sedaev, Semenov, and Sukochev, they gave examples of 
symmetric spaces 
 not being a closed
subspace of any fully symmetric space. 
Our space $X$ is already strongly symmetric under its displayed norm and is a
closed ideal in the fully symmetric space $M_\psi$.  The failure here is
exactly the failure of Hardy--Littlewood order solidity.

\begin{theorem}\label{thm:main2}
There exists a symmetric Banach function lattice $X$ on $\Rplus$ with the following properties:
\begin{enumerate}
\item If $v\in X$, $u\in L_{\log_+}(0,\infty )$, and $u\lsub v$, then
\[
  u\in X,
  \qquad
  \|u\|_X\le \|v\|_X.
\]
\item The norm on $X$ is strongly symmetric.
\item There are nonnegative decreasing functions $f\in X$ and $g\notin X$
such that
\[
  g\hsub f,
  \qquad
  \|f\|_{M_\psi}=\|g\|_{M_\psi}=1.
\]
Consequently, there exists no fully symmetric norm on $X$ which is equivalent to $\left\|\cdot\right\|_X$. 
\end{enumerate}
\end{theorem}

\section{Preliminaries}

Let $L_0(0,\infty)$ denote the vector lattice of all real- or
complex-valued Lebesgue measurable functions on $(0,\infty)$ that are
finite almost everywhere, where functions equal almost everywhere are
identified.  The order and lattice operations in $L_0(0,\infty)$ are
understood pointwise almost everywhere; in particular, the notation
$L_0(0,\infty)$ imposes no integrability condition.  For
$h\in L_0(0,\infty)$, let $h^*$ denote the decreasing rearrangement of
$|h|$, see \cite{KPS,LT2,BS}.

\begin{definition}
Let $u,v$ be measurable functions.
\begin{enumerate}
\item We write $u\hsub v$ if
\[
\int_0^t u^*(s)\,ds\le \int_0^t v^*(s)\,ds,\qquad(t>0).
\]
It is well known that $u+v\hsub u^*+v^*$ \cite{LT2,DPS,BS}.
\item Provided the positive parts of $\log u^*$ and $\log v^*$ are locally integrable, we write $u\lsub v$ if
\begin{equation}\label{eq:logsub}
  \int_0^t\log u^*(s)\,ds
  \le
  \int_0^t\log v^*(s)\,ds
  \qquad(t>0).
\end{equation}
Here, $\log0=-\infty$, and the integrals are understood as extended integrals in $[-\infty,\infty)$.
\end{enumerate}
\end{definition}
For convenience, we denote
\begin{equation}\label{eq:Kfunctional}
   \K_t(h)=\int_0^t h^*(s)\,ds,
   \qquad t>0.
 \end{equation}
In particular, we have 
 \begin{equation}\label{eq:KyFan}
   \K_t(h+k)\le \K_t(h)+\K_t(k),
   \qquad t>0.
 \end{equation}

We use the notation\cite[(3.1)--(3.4)]{DDSZ}
\[
  L_{\log_+}(0,\infty)
  =\{h:\ \log_+h^*\in (L_1+L_\infty)(0,\infty )\},
\] 
where $\log_+ t=\max\{\log t,0\}$ for $t>0$, and $\log_+0=0$;
those formulas define
the corresponding class for semifinite measurable operators and identify its
commutative form with the condition displayed above.  This is the semifinite
extension of the finite-trace class $\mathcal M^\Delta$ introduced in
\cite[Definition~2.1]{HaagerupSchultz}.  For members of this class, the extended integrals in
\eqref{eq:logsub} are well defined.

\begin{definition}\cite{KPS,DPS}\label{def:symmetric}
A Banach lattice $E\subset L_0\Rplus$ is called \emph{symmetric} if, whenever $v\in E$ and $u^*\le v^*$, one has $u\in E$ and $\|u\|_E\le\|v\|_E$.  
It is called \emph{strongly symmetric} if
\[
  u,v\in E,\quad u\hsub v
  \quad\Longrightarrow\quad
  \|u\|_E\le\|v\|_E.
\]
It is called \emph{fully symmetric} if $u\in L_0\Rplus$, $v\in E$, and
$u\hsub v$ imply that $u\in E$ and $\|u\|_E\le\|v\|_E$.
\end{definition}
Definition~\ref{def:symmetric} follows the symmetric-space convention of
\cite[Chapter~II, \S4]{KPS}; in particular, it does not assume the Fatou
property.  Separately, the norm has the \emph{Fatou property} if, whenever
$0\le h_m\uparrow h$ almost everywhere as $m\to\infty$ and
$\sup_m\|h_m\|_E<\infty$, one has $h\in E$ and
\[
  \|h_m\|_E\uparrow\|h\|_E
  \qquad\text{as }m\to\infty.
\]

The following fact can be found in \cite[Proposition~3.2]{DDSZ}; see also
\cite[Corollary~2.5]{HSZ} and
\cite[Chapter~1, Theorem~D.2]{MOA}.

\begin{lemma}\label{lem:powers}
Let $u,v\in L_{\log_+}(0,\infty )$.  If $u\lsub v$, then, for every $q>0$, we have 
\begin{equation}\label{eq:powers}
  \int_0^t u^*(s)^q\,ds
  \le
  \int_0^t v^*(s)^q\,ds
  \qquad(t>0).
\end{equation}
In particular, $u\lsub v$ implies $u\hsub v$.
\end{lemma}

Below, we define a Marcinkiewicz  space with respect to  the function 
\begin{equation}\label{eq:psi}
  \psi(t)=
  \begin{cases}
    t, & 0<t\le1,\\
    1+\log t, & t>1.
  \end{cases}
\end{equation}
The function $\psi$ is increasing and concave, with $\psi(0+)=0$.  Let
\begin{equation}\label{eq:Marcinkiewicz}
  M_\psi(0,\infty )
  =
  \left\{
    h:\ \|h\|_{M_\psi}:=
    \sup_{t>0}\frac{\int_0^t h^*(s) ds}{\psi(t)}<\infty
  \right\},
\end{equation}
see \cite[Chapter~II, \S5, formula~(5.9)]{KPS}, \cite[Section~6.4]{DPS} and \cite[p.~2]{KP}. 

\begin{proposition}\label{prop:Mpsi}
The functional in \eqref{eq:Marcinkiewicz} is a fully symmetric Banach
function norm with the Fatou property.  Moreover,
$M_\psi(0,\infty)\subset L_{\log_+}(0,\infty)$.
\end{proposition}

\begin{proof}
The first assertion is a standard fact; see, for example,
\cite[Chapter~II, \S5, formula~(5.9), pp.~112--115]{KPS},
\cite[Theorem~6.4.1]{DPS}, and \cite[p.~2]{KP}.

Observe that 
for $h\in M_\psi(0,\infty)$ and $t>0$, we have 
\[
  h^*(t)
  \le \frac1t\int_0^t h^*(s)\,ds
  \le \|h\|_{M_\psi}\frac{\psi(t)}t.
\]
Since $\psi(t)/t=(1+\log t)/t\to0$ as $t\to\infty$, it follows that
$h^*(t)\to0$ as $t\to\infty$.  On the other hand, $\psi(t)=t$ for
$0<t\le1$ gives $h^*(t)\le\|h\|_{M_\psi}$ and hence
$h\in L_\infty(0,\infty)$. 
Thus, $\log_+h^*$ is bounded and vanishes for
all sufficiently large $t$.
In particular,  $h\in L_{\log_+}(0,\infty )$.
\end{proof}

\section{Proofs of the main results}

\subsection{Construction of logarithmically monotone symmetric norms}
Let $\psi$ be defined by \eqref{eq:psi}.  Fix a parameter
$\alpha\in(0,1]$.
Set
\begin{equation}\label{eq:scales}
  L_0=0,\qquad L_1=4,\qquad
  L_n=L_{n-1}^2\quad(n\ge2),
\end{equation}
and define
\begin{equation}\label{eq:parameters}
  \begin{aligned}
    T_n&=e^{L_n}, &\qquad r_n&=L_n-L_{n-1},\\
    p_n&=1-r_n^{-\alpha}, &\qquad
    d_n&=\frac{T_n}{\psi(T_n)}=\frac{T_n}{1+L_n}.
  \end{aligned}
\end{equation}
Only $p_n$ and the subsequently defined objects $A_n$, $\Phi$,
$N_\lambda$, and $X$ depend on $\alpha$; this dependence is suppressed in
the notation.  Since $r_n\ge4$, we have
$0<p_n<1$ for every $n\ge1$.  Also,
$L_n\to\infty$, $r_n\to\infty$, and $p_n\to1$ as $n\to\infty$.  Moreover,
\begin{equation}\label{eq:scale-limits}
  \frac{r_n}{L_n}\longrightarrow1,\qquad
  \frac{L_n}{r_n}\longrightarrow1,\qquad
  \frac{L_n}{r_n-1}\longrightarrow1
  \qquad\text{as }n\to\infty.
\end{equation}

The following functional $\Phi$ is the crucial ingredient of our construction of 
a  logarithmically monotone symmetric norm on $M_\psi(0,\infty )$. 
For any  $h\in M_\psi(0,\infty)$, put
\begin{equation}\label{eq:An}
  A_n(h)=
  \left(
    \frac1{T_n}\int_0^{T_n}h^*(s)^{p_n}\,ds
  \right)^{1/p_n}
\end{equation}
and
\begin{equation}\label{eq:Phi}
  \Phi(h)=\limsup_{n\to\infty} d_n A_n(h).
\end{equation}

\begin{proposition}\label{prop:Phi}
The functional $\Phi$ is a finite continuous symmetric lattice seminorm on
$M_\psi$.  It satisfies
\begin{equation}\label{eq:PhiBound}
  0\le\Phi(h)\le\|h\|_{M_\psi},
\end{equation}
and, for $u,v\in M_\psi$, is monotone with respect to logarithmic
submajorization:
\begin{equation}\label{eq:PhiLog}
  u\lsub v\quad\Longrightarrow\quad\Phi(u)\le\Phi(v).
\end{equation}
Moreover, $\Phi(h)=0$ whenever $h$ is bounded and has support of finite measure.
\end{proposition}

\begin{proof}
Since $0<p_n<1$, Jensen's inequality for the concave function
$x\mapsto x^{p_n}$ \cite[Chapter~3, Theorem~3.3, p.~62]{Rudin} gives
\[
  A_n(h)
  \le
  \frac1{T_n}\int_0^{T_n}h^*(s)\,ds.
\]
Consequently,
\begin{equation}\label{eq:scaledJensen}
  d_nA_n(h)
  \le
  \frac{\K_{T_n}(h)}{\psi(T_n)}
  \le
  \|h\|_{M_\psi}.
\end{equation}
This proves finiteness and \eqref{eq:PhiBound}.  Homogeneity, lattice
monotonicity, and rearrangement invariance follow immediately from the
definition.

If $\|h\|_\infty\le B<\infty$ and
$|\operatorname{supp}h|\le S<\infty$, then, for every $n$,
\[
  d_nA_n(h)
  \stackrel{\eqref{eq:An}}{\le}
  \frac{B S^{1/p_n}}{1+L_n}
  T_n^{1-1/p_n}
  \stackrel{\eqref{eq:parameters}}{=}
  \frac{B S^{1/p_n}}{1+L_n}
  \exp\left(-\frac{L_n}{r_n^\alpha-1}\right).
\]
Here, $S^{1/p_n}\to S$ as $n\to\infty$, while $\exp\left(-\frac{L_n}{r_n^\alpha-1}\right) \le 1$.  Hence $d_nA_n(h)\to0$ as $n\to\infty$.
Thus, 
$\Phi$ vanishes on bounded functions having finite supports.

It remains to prove subadditivity.  For every $n\ge1$,
$|u+v|^{p_n}\le |u|^{p_n}+|v|^{p_n}$; see
\cite[(2.2)]{KPR}.  Using \eqref{eq:KyFan}, we have
\begin{align*}
  \int_0^{T_n}(u+v)^*(s)^{p_n}\,ds
  &=\K_{T_n}(|u+v|^{p_n})\\
  &\le \K_{T_n}(|u|^{p_n})+\K_{T_n}(|v|^{p_n})\\
  &=\int_0^{T_n}u^*(s)^{p_n}\,ds
    +\int_0^{T_n}v^*(s)^{p_n}\,ds .
\end{align*}
It follows that
\begin{align}
  A_n(u+v)
  \le
  \bigl(A_n(u)^{p_n}+A_n(v)^{p_n}\bigr)^{1/p_n}
  \stackrel{\mbox{\tiny \cite[(2.3)]{KPR}}}{\le}
  2^{1/p_n-1}\bigl(A_n(u)+A_n(v)\bigr).
  \label{eq:quasitriangle}
\end{align}
Noting that  $c_n=2^{1/p_n-1}\to 1$  as $n\to\infty$, and the sequences
$d_nA_n(u)$ and $d_nA_n(v)$ are bounded by
\eqref{eq:scaledJensen}, we have 
\begin{align*}
  \Phi(u+v)
  &\le
  \limsup_{n\to\infty}
  c_n\bigl(d_nA_n(u)+d_nA_n(v)\bigr)\\
  &\le \Phi(u)+\Phi(v).
\end{align*}
Thus,
$\Phi$ is a seminorm.
Furthermore,
\[
  |\Phi(u)-\Phi(v)|
  \le\Phi(u-v)
  \le\|u-v\|_{M_\psi},
\]
so $\Phi$ is continuous.

Finally, if $u\lsub v$, then Lemma~\ref{lem:powers} with $q=p_n$ gives
$A_n(u)\le A_n(v)$ for every $n$.  Taking the upper limit as
$n\to\infty$ proves \eqref{eq:PhiLog}.
\end{proof}

Fix $\lambda>0$ and define
\begin{equation}\label{eq:newnorm}
  N_\lambda(h)=\|h\|_{M_\psi}+\lambda\Phi(h),
  \qquad h\in M_\psi.
\end{equation}

\begin{corollary}\label{cor:newnorm}
The functional $N_\lambda$ is a symmetric Banach norm on $M_\psi(0,\infty )$.  It is equivalent to the standard Marcinkiewicz norm:
\begin{equation}\label{eq:equivalence}
  \|h\|_{M_\psi}
  \le N_\lambda(h)
  \le(1+\lambda)\|h\|_{M_\psi}.
\end{equation}
Moreover, if $u\in L_{\log_+}(0,\infty ),~v\in M_\psi(0,\infty ),$ and $ u\lsub v$, then 
\begin{equation}\label{eq:Nlog}
  u\in M_\psi(0,\infty) \mbox{ and }
  N_\lambda(u)\le N_\lambda(v).
\end{equation}
\end{corollary}

\begin{proof}
Proposition~\ref{prop:Phi} proves the first assertion and \eqref{eq:equivalence}
(note that the completeness of $M_\psi(0,\infty )$ with respect to $N_\lambda $ follows from \eqref{eq:equivalence} immediately).  If $u\lsub v$, then
Lemma~\ref{lem:powers} gives $u\hsub v$.  
Since $\left\|\cdot\right\|_{M_\psi}$ is fully symmetric  \cite[Theorem 6.4.1]{DPS},
it follows that $u\in M_\psi$ and $\left\|u\right\|_{M_\psi }\le \left\|v\right\|_{M_\psi }$. By \eqref{eq:newnorm}, we have $\Phi(u)\le \Phi(v)$.
Adding these two inequalities yields $ N_\lambda(u)\le N_\lambda(v)$. 
\end{proof}

Below, we construct the symmetric space $X$ satisfying the conditions in Theorem~\ref{thm:main2}. 
Define
\begin{equation}\label{eq:Xdef}
  X=\ker\Phi
  =\{h\in M_\psi(0,\infty ):\ \Phi(h)=0\},
~
  \|h\|_X=\|h\|_{M_\psi}.
\end{equation}

\begin{proposition}\label{prop:X}
The space $X$ is a closed symmetric Banach function ideal in $M_\psi(0,\infty )$ and
contains every bounded function with support of finite measure.  It is solid
and norm-monotone with respect to logarithmic submajorization: if $v\in X$,
$u\in L_{\log_+}(0,\infty )$, and $u\lsub v$, then
\begin{equation}\label{eq:Xlog}
  u\in X,
  \qquad
  \|u\|_X\le\|v\|_X.
\end{equation}
The norm on $X$ is strongly symmetric.
\end{proposition}

\begin{proof}
Continuity of $\Phi$ (see Proposition \ref{prop:Phi}) shows that its kernel is a closed linear subspace of
$M_\psi(0,\infty )$; hence $X$, with the restricted  norm $\left\|\cdot\right\|_{M_\psi}$, is complete.
If $|u|\le|v|$ and $v\in X$,
then the lattice monotonicity gives
$0\le\Phi(u)\le\Phi(v)=0$; hence $u\in X$.  
Moreover, if $u^*\le v^*$ and $v\in X$, then
$v^*\in X$, the ideal property gives $u^*\in X$, and equimeasurability
then gives $u\in X$.  The symmetricity of the standard Marcinkiewicz norm gives
$\|u\|_X=\|u\|_{M_\psi}\le\|v\|_{M_\psi}=\|v\|_X$,
i.e., $X$ is symmetric.
Proposition~\ref{prop:Phi}
also gives that all bounded finite-support functions, and in particular all
indicators of finite-measure sets belong to $X$.

Suppose now that $v\in X$ and $u\lsub v$.  Lemma~\ref{lem:powers} with
$q=1$ gives $u\hsub v$.  The full symmetricity\cite[Theorem 6.4.1]{DPS}  of $M_\psi(0,\infty )$ gives that 
\[
  u\in M_\psi(0,\infty ) \mbox{ and }
  \|u\|_{M_\psi}\le\|v\|_{M_\psi}.
\]
Proposition~\ref{prop:Phi} gives
$0\le\Phi(u)\le\Phi(v)=0$, so $u\in X$ and \eqref{eq:Xlog} follows.

Finally, if $u,v\in X$ and $u\hsub v$, the  full symmetricity of $\left\|\cdot\right\|_{M_\psi}$
gives $\|u\|_X\le\|v\|_X$.  Thus, the restricted
norm is strongly symmetric.
\end{proof}

\subsection{First example: failure of strong symmetry}
For this subsection, choose $\alpha=1$ in \eqref{eq:parameters}, i.e., 
\[
  p_n=1-r_n^{-1}\qquad(n\ge1),
\]
and all the symbols $A_n$, $\Phi$, and $N_\lambda$ refer to this choice.
Below, we show that the symmetric norm $N_\lambda$ constructed in
\eqref{eq:newnorm} is not strongly symmetric.
Define a decreasing function
\begin{equation}\label{eq:fdef}
  f(t)=
  \begin{cases}
    1, & 0<t\le1,\\
    t^{-1}, & t>1.
  \end{cases}
\end{equation}
Then, we have 
\begin{equation}\label{eq:Kf}
  \K_t(f)=\psi(t),
  \qquad
  \|f\|_{M_\psi}=1.
\end{equation}

Put $T_0=1$ and, for $n\ge1$, let
\begin{equation}\label{eq:blocks}
  I_n=(T_{n-1},T_n],\qquad
  \Delta_n=T_n-T_{n-1},\qquad
  a_n=\frac{r_n}{\Delta_n}.
\end{equation}
In particular, 
$a_n$ is the average of $t^{-1}$ on $I_n$.  Define
\begin{equation}\label{eq:gdef}
  g(t)=
  \begin{cases}
    1, & 0<t\le1,\\
    a_n, & t\in I_n,\quad n\ge1.
  \end{cases}
\end{equation}

\begin{lemma}\label{lem:HLPpair}
The function $g$ is decreasing and
\begin{equation}\label{eq:HLPpair}
  g\hsub f,
  \qquad
  \|g\|_{M_\psi}=\|f\|_{M_\psi}=1.
\end{equation}
\end{lemma}
\begin{proof}
Since $t\mapsto t^{-1}$ is decreasing, it follows that 
\[
  \frac1{T_{n-1}}\ge a_n\ge\frac1{T_n}\ge a_{n+1}, ~\forall ~
n\ge1.
\]
Thus $g$ is decreasing, i.e.,  $g=g^*$.  
Moreover, we have   $g\hsub f$ \cite[Lemma~3.6.2]{LSZ}.
For every $n\ge 1$, we have 
\begin{equation}\label{eq:endpoint-equality}
  \K_{T_n}(g)
  =1+\sum_{j=1}^n a_j\Delta_j
  =1+\sum_{j=1}^n r_j
  =1+L_n
  =\K_{T_n}(f).
\end{equation}
blocks $I_n$.  
This together with 
\eqref{eq:Kf} shows that  $\|g\|_{M_\psi}=\|f\|_{M_\psi}= 1$. 
\end{proof}

We compute the auxiliary seminorms on the two functions $f$ and $g$.

\begin{proposition}\label{prop:Phi-values-alpha-one}
For the functions in \eqref{eq:fdef} and \eqref{eq:gdef}, we have 
\begin{equation}\label{eq:Phi-values-alpha-one}
  \Phi(f)=1-e^{-1},
  \qquad
  \Phi(g)=1.
\end{equation}
\end{proposition}

\begin{proof}
Note that  $1-p_n=1/r_n$.  By \eqref{eq:parameters}, for every
$n\ge1$, we have 
\begin{equation}\label{eq:fp-integral}
  \int_0^{T_n} f(t)^{p_n}\,dt
  =1+\frac{T_n^{1-p_n}-1}{1-p_n}
  =1+r_n\bigl(e^{L_n/r_n}-1\bigr).
\end{equation}
Consequently, \eqref{eq:parameters}, \eqref{eq:An}, and
$1/p_n=r_n/(r_n-1)$ give
\begin{equation}\label{eq:scaled-f}
  d_nA_n(f)
  =
  \frac{e^{-L_n/(r_n-1)}}{1+L_n}
  \left[1+r_n\bigl(e^{L_n/r_n}-1\bigr)\right]^{r_n/(r_n-1)}.
\end{equation}
By \eqref{eq:scale-limits}, we have
\[
  e^{-L_n/(r_n-1)}\longrightarrow e^{-1}
  \qquad\text{as }n\to\infty.
\]
Moreover, since $r_n/(1+L_n)\to1$ and $L_n/r_n\to1$ as
$n\to\infty$, it follows that 
\[
  \frac{1+r_n(e^{L_n/r_n}-1)}{1+L_n}
  =
  \frac1{1+L_n}
  +\frac{r_n}{1+L_n}\bigl(e^{L_n/r_n}-1\bigr)
  \longrightarrow e-1
  \qquad\text{as }n\to\infty.
\] Finally, since 
$L_n/r_n\le2$ and $r_n-1\ge L_n/2$ for all sufficiently large $n$, it follows that 
\begin{align*}
0
&\le
\log\left(
  \left[1+r_n(e^{L_n/r_n}-1)\right]^{1/(r_n-1)}
\right)\\
&=
\frac{\log\left[1+r_n(e^{L_n/r_n}-1)\right]}{r_n-1}\\
&\le
\frac{2\log\left[1+(e^2-1)L_n\right]}{L_n}
\longrightarrow0
\qquad\text{as }n\to\infty.
\end{align*}
It follows that
\[
  \left[1+r_n(e^{L_n/r_n}-1)\right]^{1/(r_n-1)}
  \longrightarrow1
  \qquad\text{as }n\to\infty.
\]
Now, we have   
\begin{align*}
  d_nA_n(f)
  &\stackrel{ \eqref{eq:scaled-f}}{=}
  e^{-L_n/(r_n-1)}
  \frac{1+r_n(e^{L_n/r_n}-1)}{1+L_n}\times
  \left[1+r_n(e^{L_n/r_n}-1)\right]^{1/(r_n-1)}\\
  &\longrightarrow e^{-1}(e-1)=1-e^{-1}
  \quad\text{as }n\to\infty.
\end{align*}
By definition, we have  $\Phi(f)=1-e^{-1}$.

Since $g=g^*$ (by Lemma~\ref{lem:HLPpair}), it follows that 
\[
  \int_0^{T_n}g(t)^{p_n}\,dt
  \ge \Delta_n a_n^{p_n}.
\]
Since
\[
  \frac{\Delta_n}{T_n}=1-e^{-r_n},
  \qquad
  \frac1{p_n}-1=\frac1{r_n-1},
\]
it follows that 
\begin{align}\label{eq:g-lower-alpha-one}
  d_nA_n(g)
  &\ge
  \frac{T_n}{1+L_n}
  \left(\frac{\Delta_n}{T_n}\right)^{1/p_n}a_n
  \notag\\
  &=
  \frac{r_n}{1+L_n}
  \left(\frac{\Delta_n}{T_n}\right)^{1/p_n-1}
  \notag\\
  &=
  \frac{r_n}{1+L_n}
  \left(1-e^{-r_n}\right)^{1/(r_n-1)}
  \stackrel{\eqref{eq:scale-limits}}{\longrightarrow}1
  \qquad\text{as }n\to\infty.
\end{align}
Indeed, by 
$  
  \frac{\log(1-e^{-r_n})}{r_n-1}\longrightarrow0
  \text{ as }n\to\infty,
$
we have $\left(1-e^{-r_n}\right)^{1/(r_n-1)}\to 1$ as $n\to \infty $.  On the other hand,
\eqref{eq:scaledJensen} and \eqref{eq:endpoint-equality} give
\begin{align}\label{dnAng}
  d_nA_n(g)
  \le
  \frac{\K_{T_n}(g)}{\psi(T_n)}
  =1.
\end{align}
Therefore, 
$d_nA_n(g)\to1$ as $n\to\infty$, and hence $\Phi(g)=1$.
\end{proof}

\begin{proof}[Proof of Theorem~\ref{thm:main}]
Take $N=N_\lambda$ for an arbitrary $\lambda>0$.  By
Corollary~\ref{cor:newnorm}, it is a symmetric Banach norm and is
logarithmically monotone.  Lemma~\ref{lem:HLPpair} and
Proposition~\ref{prop:Phi-values-alpha-one} give
\[
  g\hsub f,
\]
but
\[
  N_\lambda(g)=1+\lambda
  >1+\lambda(1-e^{-1})
  =N_\lambda(f).
\]
Thus,  $N_\lambda$ is not strongly symmetric and, a fortiori, is not fully symmetric.
In particular, it fails the Fatou property; 
see, e.g. \cite[Corollary 5.1.11]{DPS} or \cite{LT2}. 
\end{proof}

\subsection{Second example: failure of Hardy--Littlewood solidity}
For the second main theorem, let  $\alpha=\frac12$. In particular,  
\[
  p_n=1-r_n^{-1/2}\qquad(n\ge1),
\]
and the symbols $A_n$, $\Phi$, and $X$ now refer to this choice.  The
scales $L_n,T_n,r_n$ and the functions $f,g$ are unchanged.

\begin{proposition}\label{prop:Phi-values-alpha-half}
For the functions $f$ and $g$ defined above, we have 
\begin{equation}\label{eq:Phi-values-alpha-half}
  \Phi(f)=0,
  \qquad
  \Phi(g)=1.
\end{equation}
\end{proposition}

\begin{proof}
Put
\[
  \delta_n=1-p_n=\frac1{\sqrt{r_n}}.
\]
Since $0<\delta_n<1$ for every $n\ge1$, it follows that  
\begin{align*}
  \int_0^{T_n}f(t)^{p_n}\,dt
  =1+\frac{T_n^{\delta_n}-1}{\delta_n}
  =\frac{T_n^{\delta_n}+\delta_n-1}{\delta_n}
  \le\delta_n^{-1}T_n^{\delta_n}.
\end{align*}
Since $1+(\delta_n-1)/p_n=0$, it follows that 
\begin{align}
  d_nA_n(f)
  \le
  \frac{T_n}{1+L_n}
  \left(\delta_n^{-1}T_n^{\delta_n-1}\right)^{1/p_n}
  =\frac{\delta_n^{-1/p_n}}{1+L_n}
  =\frac{r_n^{1/(2p_n)}}{1+L_n}.
  \label{eq:fupper}
\end{align}
For all sufficiently large $n$, we have 
$p_n\ge3/4$; since $r_n\le L_n$, the right-hand side of 
\eqref{eq:fupper} is bounded above by
\[
  \frac{L_n^{2/3}}{1+L_n}
  \longrightarrow0
  \qquad\text{as }n\to\infty.
\]
Thus, $d_nA_n(f)\to0$ as $n\to\infty$, and $\Phi(f)=0$.

Since $g=g^*$ (by Lemma~\ref{lem:HLPpair}), it follows that 
\[
  \int_0^{T_n}g(t)^{p_n}\,dt
  \ge\Delta_n a_n^{p_n}.
\]
Since
\[
  \frac{\Delta_n}{T_n}=1-e^{-r_n},
  \qquad
  \frac1{p_n}-1=\frac1{\sqrt{r_n}-1},
\]
it follows that 
\begin{align}
  d_nA_n(g)
  &\ge
  \frac{T_n}{1+L_n}
  \left(\frac{\Delta_n a_n^{p_n}}{T_n}\right)^{1/p_n}
  \notag\\
  &=
  \frac{r_n}{1+L_n}
  \left(\frac{\Delta_n}{T_n}\right)^{1/p_n-1}
  \notag\\
  &=
  \frac{r_n}{1+L_n}
  \left(1-e^{-r_n}\right)^{1/(\sqrt{r_n}-1)}
  \longrightarrow1
  \qquad\text{as }n\to\infty,
  \label{eq:g-lower-alpha-half}
\end{align}
because $r_n/(1+L_n)\to1$ as $n\to\infty$ and
\[
  \frac{\log(1-e^{-r_n})}{\sqrt{r_n}-1}
  \longrightarrow0
  \qquad\text{as }n\to\infty.
\]
On the other hand, \eqref{eq:scaledJensen} and
\eqref{eq:endpoint-equality} give
\[
  d_nA_n(g)
  \le\frac{\K_{T_n}(g)}{\psi(T_n)}=1.
\]
Thus $d_nA_n(g)\to1$ as $n\to\infty$, proving $\Phi(g)=1$.
\end{proof}

\begin{proof}[Proof of Theorem~\ref{thm:main2}, parts (1)--(3)]
Proposition~\ref{prop:X} gives the logarithmic solidity, logarithmic norm
monotonicity, and strong symmetricity.
Proposition~\ref{prop:Phi-values-alpha-half} gives
\[
  f\in X,
  \qquad
  g\notin X,
\]
while Lemma~\ref{lem:HLPpair} gives $g\hsub f$. 
Hence $X$ is not solid under Hardy--Littlewood submajorization. 
In particular,   no norm equivalent
to $\left\|\cdot\right\|_X$ can make $X$ fully symmetric.
\end{proof}


\end{document}